\documentclass[11pt]{amsart}
\usepackage{tikz-cd}
\usepackage{verbatim}
\usepackage[colorinlistoftodos]{todonotes}
\usepackage{mathtools}
\usepackage[bookmarks,colorlinks,breaklinks]{hyperref}  
\usepackage[top=2cm, bottom=4.5cm, left=3cm, right=3cm]{geometry}
\hypersetup{linkcolor=blue,citecolor=blue,filecolor=blue,urlcolor=blue} 
\newtheorem{thm}{Theorem}[section]
\newtheorem{proposition}[thm]{Proposition}

\newtheorem{defn}[thm]{Definition}

\theoremstyle{plain}
\numberwithin{equation}{theorem}

\theoremstyle{remark}
\newtheorem{rem}[thm]{Remark}
\newtheorem{ex}[thm]{Example}
\newcommand{\F}{{\mathbb F}}

\newcommand{\bG}{{\mathbb G}}

\newcommand{\Fp}{\mathbb F_p}
\newcommand{\Fq}{\mathbb F_q}

\newcommand{\lra}{\longrightarrow}

\title[The Mordell-Lang conjecture in prime characteristic]{The Mordell-Lang conjecture for semiabelian varieties defined over fields of positive characteristic}

\author{Dragos Ghioca}
\address{Department of Mathematics \\ University of British Columbia \\ 1984 Mathematics Road \\ Canada V6T 1Z2}
\email{dghioca@math.ubc.ca}

\author{She Yang}
\address{Beijing International Center for Mathematical Research \\ Peking University \\ Beijing 100871 \\ China}
\email{ys-yx@pku.edu.cn}


\begin{document}

\begin{abstract}
Let $G$ be a semiabelian variety defined over  an algebraically closed  field $K$ of prime characteristic. 
We  describe the intersection of a subvariety $X$ of $G$ with a finitely generated subgroup of $G(K)$.  
\end{abstract}

\maketitle


\section{Introduction}

The purpose of this note is to prove a variant of the Mordell-Lang conjecture for semiabelian varieties defined over fields of positive characteristic. More precisely, let $G$ be a semiabelian variety defined over an algebraically closed field $K$, i.e., there exists a short exact sequence of algebraic groups defined over $K$:
\begin{equation}
\label{eq:0}
1\lra \mathbb{G}_m^N\lra G\lra A\lra 1,
\end{equation}
where $N\ge 0$ is an integer and $A$ is an abelian variety. Assuming $K$ has characteristic $p>0$, then for any subvariety $X\subseteq G$ defined over $K$ and any finitely generated subgroup $\Gamma\subset G(K)$, we describe the intersection $X(K)\cap\Gamma$. In particular, we fix an error in the paper \cite{TAMS} of the first author where a simplified form of the aforementioned was claimed in the case $G$ is defined over a finite subfield of $K$; we present several examples showing that the intersection $X(K)\cap\Gamma$ involves the more general $F$-sets appearing in Definition~\ref{def:F1}.


\subsection{General background}

The Mordell-Lang conjecture for semiabelian varieties $G$ defined over fields of characteristic $0$ predicts that the intersection of a subvariety $X\subseteq G$ with a finitely generated subgroup $\Gamma$ of $G$ is a finite union of cosets of subgroups of $\Gamma$. This conjecture was proven by Laurent \cite{Laurent} in the case of tori, Faltings \cite{Faltings} in the case of abelian varieties, and by  Vojta \cite{Vojta} in the general case of semiabelian varieties. In particular, their results show that if $X$ is an irreducible subvariety of $G$ which intersects a finitely generated group in a Zariski dense subset, then $X$ must be a translate of a semiabelian subvariety of $G$.

The picture for positive characteristic fields $K$ is more complicated due to the existence of the Frobenius endomorphism for varieties defined over finite fields; in particular, it is no longer true that only translates of semiabelian subvarieties of $G$ have the property that they intersect a finitely generated subgroup of $G$ in a Zariski dense subset. Hrushovski \cite{Hrushovski} obtained the right shape for the irreducible subvarieties $X$ whose intersection with a finitely generated subgroup $\Gamma$ is Zariski dense. 
\begin{thm}[Hrushovski \cite{Hrushovski}]
\label{thm:H}
Let $G$ be a semiabelian variety defined over an algebraically closed field $K$ of characteristic $p$. Let $\Gamma\subset G(K)$ be a finitely generated subgroup and let $X\subseteq G$ be an irreducible subvariety with the property that $X(K)\cap\Gamma$ is Zariski dense in $X$. Then there exists $\gamma\in G(K)$, there exists a semiabelian subvariety $G_0\subseteq G$ defined over $K$, there exists a semiabelian variety $H$ along with a subvariety $X_0\subseteq H$ both defined over a finite subfield $\Fq$ of $K$, and there exists a surjective group homomorphism $h:G_0\lra H$ such that $X=\gamma+h^{-1}(X_0)$.
\end{thm}

However, \cite{Hrushovski} left open the description of the actual intersection between the subvariety $X$ and the group $\Gamma$; next, we will address exactly this issue.


\subsection{The case of semiabelian varieties defined over finite fields and of finitely generated subgroups invariant under the Frobenius endomorphism}

Essentially, Hrushovski's result (see Theorem~\ref{thm:H}) reduced the description of the intersection $X(K)\cap\Gamma$ to the case when the ambient semiabelian variety is defined over a finite field. Moosa and Scanlon \cite{F-sets, F-sets_2} addressed precisely this problem under an additional assumption on the subgroup $\Gamma$; in order to state their main result, we introduce a little bit of notation.

\begin{defn}
\label{def:F0}
For a semiabelian variety $G$ defined over a finite subfield $\Fq$ of an algebraically closed field $K$ of characteristic $p$, we define a \emph{groupless $F$-set} any subset of $G(K)$ of the form:
\begin{equation}
\label{eq:F-orbits}
\left\{\alpha_0+\sum_{i=1}^r F^{kn_i}(\alpha_i)\colon n_i\in\mathbb{N}\right\},
\end{equation}
where $r\ge 0$, $\alpha_0,\alpha_1,\dots, \alpha_r\in G(K)$ and $k\in\mathbb{N}$, while $F$ is the Frobenius endomorphism of $G$ corresponding to  the finite field $\Fq$. 

For any finitely generated subgroup $\Gamma\subset G(K)$,  we define a \emph{groupless $F$-set in $\Gamma$} as a groupless $F$-set contained in $\Gamma$. Also, an \emph{$F$-set in $\Gamma$} is any set of the form $S+B$, where $S$ is a groupless $F$-set in $\Gamma$ and $B$ is a subgroup of $\Gamma$ (as always, for any two subsets $B$ and $C$ of $G$, we have that $C+B$ is simply the set of all $c+b$ where $b\in B$ and $c\in C$).
\end{defn}

\begin{rem}
\label{rem:same_k}
In \cite[Theorem B]{F-sets}, Moosa and Scanlon allowed for the possibility that a groupless $F$-set involves sums of $F$-orbits as in equation~\eqref{eq:F-orbits} of the form
\begin{equation}
\label{eq:F-set_0}
\alpha_0+\sum_{i=1}^r F^{k_in_i}(\alpha_i)\text{ (as $n_i$ vary in $\mathbb{N}$),}
\end{equation}
for given, but potentially distinct, positive integers $k_i$. However, each  $F$-set from equation~\eqref{eq:F-set_0} is a union of finitely many $F$-sets given as in Definition~\ref{def:F0} (simply by working with $k$ as the least common multiple of $k_1,\dots, k_r$).
\end{rem}

\begin{thm}[Moosa-Scanlon \cite{F-sets}]
\label{thm:R-T}
Let $G$ be a semiabelian variety defined over a finite subfield $\Fq$ of an algebraically closed field $K$ and let $F:G\lra G$ be the Frobenius endomorphism associated to the finite field $\Fq$. Let $X\subseteq G$ be a subvariety defined over $K$ and let $\Gamma\subset G(K)$ be a finitely generated subgroup. If $\Gamma$ is invariant under  $F^\ell$ for some $\ell\in\mathbb{N}$, then $X(K)\cap\Gamma$ is a finite union of $F$-sets in $\Gamma$. 
\end{thm}


\subsection{The case of an arbitrary finitely generated subgroup}

It is natural to ask whether the above description from Theorem~\ref{thm:R-T} of the intersection $X(K)\cap\Gamma$ remains valid also when $\Gamma$ is no longer invariant under a power of the Frobenius endomorphism of $G$ (but only assume $\Gamma$ is finitely generated). 

One could consider the $\mathbb{Z}[F]$-submodule $\tilde{\Gamma}\subset G(K)$ spanned by $\Gamma$ and since $F$ is integral over $\mathbb{Z}$ (seen as a subring of ${\rm End}(G)$), then $\tilde{\Gamma}$ is still finitely generated and so, Moosa-Scanlon's result (see Theorem~\ref{thm:R-T}) yields that $X(K)\cap\tilde{\Gamma}$ is a finite union of $F$-sets in $\tilde{\Gamma}$. So, the problem reduces to understanding the intersection of an $F$-set $S$ in $\tilde{\Gamma}$ with $\Gamma$. The first author \cite[Theorem~3.1]{TAMS} proved that when $S$ is a \emph{groupless} $F$-set in $\tilde{\Gamma}$, then its intersection with $\Gamma$ is a finite union of groupless $F$-sets in $\Gamma$. Also, the first author analyzed in \cite{TAMS} the intersection with $\Gamma$ of an arbitrary $F$-set in $\tilde{\Gamma}$; however, the final assertion from \cite[Step~3,~p.~3842]{TAMS} claiming that the general case of an $F$-set reduces to the groupless case is not valid, as shown by the constructions from Section~\ref{sec:examples} (see Examples~\ref{ex:1}~and~\ref{ex:2} which were found by the second author). Essentially, the error from \cite{TAMS} was to claim that the pullback of a groupless $F$-set in $\tilde{\Gamma}$ through a group homomorphism restricted to $\Gamma$ must be an $F$-set in $\Gamma$ (as in Definition~\ref{def:F0}). Furthermore, Example~\ref{ex:3} shows that when $\Gamma$ is an arbitrary finitely generated subgroup, the intersection $X(K)\cap\Gamma$ can be quite wild; this motivates our Definition~\ref{def:F1}  which yields the right form of the sets appearing in the intersection of a subvariety of $G$ with a finitely generated group.

\begin{defn}
\label{def:F1}
For a semiabelian variety $G$ defined over a finite subfield $\Fq$ of an algebraically closed field $K$ of characteristic $p$ and a finitely generated subgroup $\Gamma\subset G(K)$, we define a \emph{generalized $F$-set in $\Gamma$} any subset of $\Gamma$ of the form:
\begin{equation}
\label{eq:pi}
\left(\pi|_{\Gamma}\right)^{-1}(S),
\end{equation}
where $\pi:G\lra H$ is a surjective group homomorphism of semiabelian varieties both defined over a finite subfield of $K$ for which $\dim(\ker(\pi))>0$, $\pi|_\Gamma$ is its restriction to the subgroup $\Gamma$, and $S\subset H(K)$ is a groupless $F$-set in $\pi(\Gamma)$. 

Note that $H$ may be defined over another finite subfield $\F_{q'}\subset K$ and thus the set $S$ from equation~\eqref{eq:pi} is a groupless $F$-set in $\pi(\Gamma)$ where $F$ stands for the Frobenius endomorphism of $H$ associated to the finite field $\F_{q'}$.
\end{defn}


\subsection{Our results}
Now we can state our main results, first for describing the intersection with a finitely generated group of a subvariety of a semiabelian variety defined over a finite field.

\begin{thm}
\label{thm:main}
Let $G$ be a semiabelian variety defined over a finite subfield $\Fq$ of an algebraically closed field $K$ of characteristic $p$. Let $X\subset G$ be a subvariety defined over $K$ and let $\Gamma\subset G(K)$ be a finitely generated subgroup. Then the intersection $X(K)\cap\Gamma$ is a  union of finitely many groupless $F$-sets in $\Gamma$ along with finitely many generalized $F$-sets in $\Gamma$. 
\end{thm}

Our Examples~\ref{ex:1},~\ref{ex:2}~and~\ref{ex:3} show that the sets appearing as intersections between a subvariety $X$ of a semiabelian variety $G$ defined over a finite field with a finitely generated subgroup can be quite complicated, well-beyond the world of $F$-sets from Definition~\ref{def:F0}. However, when  $X$ is a curve or $G$ is a simple semiabelian variety (i.e., either a simple abelian variety or a $1$-dimensional torus), then we can show that the intersection $X(K)\cap\Gamma$ is a finite union of $F$-sets in $\Gamma$.

\begin{thm}
\label{thm:3}
Let $G$ be a semiabelian variety defined over a finite subfield of an algebraically closed field $K$ of prime characteristic, let $X\subseteq G$ be a subvariety defined over $K$ and let $\Gamma\subset G(K)$ be a finitely generated subgroup. If either $\dim(X)=1$ or $G$ is a simple semiabelian variety, then $X(K)\cap\Gamma$ is a finite union of $F$-sets.
\end{thm}

Next, combining our Theorem~\ref{thm:main} with Hrushovski's result (see Theorem~\ref{thm:H}), we obtain the description of the intersection of a subvariety of an arbitary semiabelian variety $G$ defined over a field of prime characteristic with a finitely generated subgroup of $G$. For this end we introduce the notion of \emph{pseudo-generalized $F$-sets}.

\begin{defn}
\label{def:F2}
Let $G$ be a semiabelian variety defined over an algebraicaly closed field $K$ of characteristic $p$ and let $\Gamma\subset G(K)$ be a finitely generated subgroup. A \emph{pseudo-generalized $F$-set in $\Gamma$} is a set of the form 
$$x_0+\left(\pi|_{\Gamma_0}\right)^{-1}(S),$$
where $x_0\in\Gamma$, $G_0\subseteq G$ is a semiabelian subvariety, $\Gamma_0=G_0(K)\cap\Gamma$, $H$ is a semiabelian variety defined over a finite subfield $\F_q\subset K$, $\pi:G_0\lra H$ is a surjective group homomorphism of semiabelian varieties, and $S\subset H(K)$ is a groupless $F$-set in $\pi(\Gamma_0)$.  
\end{defn}

\begin{rem}
\label{rem:less}
In Definition~\ref{def:F2}, if $G$ is defined over a finite subfield of $K$, then the pseudo-generalized $F$-sets from Definition~\ref{def:F2} cover both the  groupless $F$-sets in $\Gamma$ from Definition~\ref{def:F0} and also the generalized $F$-sets in $\Gamma$ from Definition~\ref{def:F1}, but it is a bit more general than those two types of sets. 
\end{rem}

\begin{thm}
\label{thm:main_2}
Let $G$ be a semiabelian variety defined over an algebraicaly closed field $K$ of characteristic $p$, let $X\subseteq G$ be a subvariety and let $\Gamma\subset G(K)$ be a finitely generated group. Then $X(K)\cap\Gamma$ is a finite union of pseudo-generalized $F$-sets in $\Gamma$.
\end{thm}


\subsection{Plan for our paper}

In Section~\ref{sec:examples} we introduce three examples which progressively show the complexity of the sets appearing as intersections between a subvariety of a semiabelian variety $G$ with a finitely generated group. Even though in our examples, $G$ is defined over a finite field, each such example can be ``embedded'' as isotrivial semiabelian subvarieties of a semiabelian variety defined over an arbitrary field of positive characteristic, thus  providing complex examples of pseudo-generalized $F$-sets. In Section~\ref{sec:proofs} we prove Theorems~\ref{thm:main}~and~\ref{thm:main_2}. Also, we prove Theorem~\ref{thm:3} as a consequence of two more precise results (see Propositions~\ref{prop:curve}~and~\ref{prop:simple}) regarding the structure of the intersection $X(K)\cap \Gamma$ when either $X$ is a curve, or $G$ is a simple semiabelian variety.


\section{Examples}
\label{sec:examples}

Our first example already shows that $X(K)\cap\Gamma$ is not always an $F$-set in $\Gamma$ (when $\Gamma$ is not invariant under the Frobenius endomorphism of $G$).
\begin{ex}
\label{ex:1}
We let $G=\bG_m^2\times E$, where $E$ is a supersingular elliptic curve defined over $\Fp$; for example, we can take $E$ be the elliptic curve given by the  equation in affine coordinates $y^2=x^3+1$ when $p=5$, in which case, we have that the square $F^2$ of the usual Frobenius endomorphism of $E$ corresponding to $\mathbb{F}_5$ equals the multiplication map $[-5]$ on $E$. We let $C\subset \bG_m^2$ be the line given by the equation $x_2=x_1+1$ and then let $X=C\times E$. We let $K=\overline{\Fp(t)}$ and let $P\in E(K)$ be a nontorsion point. Finally, we let $\Gamma\subset G(K)$ be the cyclic group spanned by $Q:=(t,t+1,P)\subset G(K)$. Then 
\begin{equation}
\label{eq:ex-1}
X(K)\cap\Gamma=\left\{ p^nQ\colon n\ge 0\right\}.
\end{equation}
Furthermore, the set from~\eqref{eq:ex-1} cannot be expressed as a groupless $F$-set; the closest it comes to being an $F$-set is expressing it as the following slight twist of groupless $F$-sets. We let $Q_1:=(t,t+1,0)\in G(K)$ and $Q_2:=(1,1,P)\in G(K)$ and then the set from~\eqref{eq:ex-1} is the union of the following two sets:
\begin{equation}
\label{eq:ex-12}
\left\{F^{2n}(Q_1)+F^{4n}(Q_2)\colon n\ge 0\right\}\text{ and }\left\{F^{2n+1}(Q_1) - F^{4n+2}(Q_2)\colon n\ge 0\right\}.
\end{equation}
\end{ex}

Now, comparing the sets from~\eqref{eq:ex-12} with the actual (groupless) $F$-sets, the difference seems quite small and so, one might think that perhaps slightly extending the definition of $F$-sets as in equation~\eqref{eq:ex-12} would be enough. The main issue in Example~\ref{ex:1} comes from the fact that the Frobenius endomorphism has ``different weights'' on the abelian, respectively affine part of $G$; so, it might be reasonable for one to think that  allowing different weights also in the definition of a groupless $F$-set by considering sets of the form:
$$\left\{\sum_{i=1}^r\sum_{j=1}^s F^{k_{i,j}\cdot n_j}(\alpha_j)\colon n_j\ge 0\text{ for }j=1,\dots, s\right\}$$
would suffice for describing $X(K)\cap\Gamma$.  
However, the next example shows that no simple extension of the definition of $F$-sets would work.

\begin{ex}
\label{ex:2}
We still work with $G=\bG_m^2\times E$, but this time the elliptic curve $E$ is ordinary; for example, we could take $p=5$ and let $E$ be the elliptic curve given by the equation in affine coordinates $y^2=x^3+x$. One can check that the  Frobenius endomorphism corresponding to $\F_5$  satisfies the integral equation $F^2-2F+5=0$ on $E$. We let as before $K=\overline{\Fp(t)}$ and we work with the cyclic group $\Gamma$  spanned by $Q:=(t,t+1,P)\in G(K)$ for some non-torsion point $P\in E(K)$. Then letting $X=C\times E$, where $C\subset \bG_m^2$ is the line $x_2=x_1+1$, we get that 
\begin{equation}
\label{eq:ex-2}
X(K)\cap\Gamma=\left\{p^nQ\colon n\ge 0\right\}.
\end{equation} 
However, one can show that the set from~\eqref{eq:ex-2} \emph{cannot} be split into finitely many sets of the form:
\begin{equation}
\label{eq:ex-22}
\left\{\sum_{i=1}^r\sum_{j=1}^s F^{k_{i,j}n_j}(Q_j)\colon n_j\ge 0\text{ for }j=1,\dots, s\right\},
\end{equation}
for any given $r,s\in\mathbb{N}$ and any choice of non-negative integers $k_{i,j}$ and any choice of given points $Q_j\in G(K)$. In other words, even the most complex definition of a groupless $F$-set as in equation~\eqref{eq:ex-22} would still not cover a possible intersection $X(K)\cap\Gamma$. 
\end{ex}

Now, Examples~\ref{ex:1}~and~\ref{ex:2} may still suggest that the intersection $X(K)\cap\Gamma$ could be expressed using more general (groupless) $F$-sets in which one would allow also the multiplication-by-$p$ map on $G$ playing a similar role as the Frobenius endomorphism. However, the next example shows that $X(K)\cap\Gamma$ may have a very complex structure.

\begin{ex}
\label{ex:3}
We let $A$ and $B$ be semiabelian varieties defined over a finite subfield $\Fq$ of an algebraically closed field $K$, let $G=A\times B$, and let $F$ be the corresponding Frobenius endomorphism associated to $\Fq$. We let $h$ be the minimal (monic) polynomial with integer coefficients for which $h(F)=0$ on $B$. Depending on the abelian part of the semiabelian variety $B$, the degree $m$ of the polynomial $h$ may be arbitrarily large.

We let $C\subset B$ be a curve defined over  $\Fq$ with trivial stabilizer in $B$ and let $P\in C(K)$ be a non-torsion point; one can even choose $C$ and $P$ so that $C(K)$ intersects the cyclic $\mathbb{Z}[F]$-module $\Gamma_1$ spanned by $P$ precisely in the orbit of $P$ under the Frobenius endomorphism $F$. We also let $Q_1,\dots, Q_m\in A(K)$ be linearly independent points (note that $A(K)\otimes_{\mathbb{Z}}\mathbb{Q}$ is an infinite dimensional $\mathbb{Q}$-vector space). Then we consider $X:=A\times C$ and also, consider the group $\Gamma\subset G(K)$ spanned by the points: 
$$R_1:=(Q_1,P)\text{; }R_2:=(Q_2,F(P))\text{; }R_3:=\left(Q_3,F^2(P)\right)\text{; }\cdots\text{; }R_m:=\left(Q_m,F^{m-1}(P)\right).$$
Then letting $\pi_2:G\lra B$ be the projection of $G=A\times B$ on the second coordinate, we have that $\pi_2(\Gamma)=\Gamma_1$ because $\Gamma_1$ is spanned by the points 
$$P\text{, }F(P)\text{, }F^2(P),\cdots, F^{m-1}(P)\in B(K)$$
since $\Gamma_1$ is the cyclic $\mathbb{Z}[F]$-module spanned by $P$ and $h(F)(P)=0$.    
So, we  can find $m$  sequences $\left\{a^{(i)}_{n}\right\}_{n\ge 0}$ of integers (for $i=0,\dots, m-1$) such that for any $n\ge 0$, we have that
\begin{equation}
\label{eq:ex-3}
F^n(P)=\sum_{i=0}^{m-1}a^{(i)}_n\cdot F^i(P).
\end{equation}
Equation~\eqref{eq:ex-3} yields that $X(K)\cap\Gamma$ is the set:
\begin{equation}
\label{eq:ex-32}
\left\{\sum_{i=1}^{m} a^{(i-1)}_n\cdot R_i\colon n\ge 0\right\}.
\end{equation}
So, due to the potential complexity of the coefficients of the polynomial $h$ satisfied by the Frobenius endomorphism (on the semiabelian variety $B$), the  sequences $\left\{a^{(i)}_n\right\}_{n\ge 0}$ may be quite complicated. 
\end{ex}


\section{Proofs of our main results}
\label{sec:proofs}

\begin{proof}[Proof of Theorem~\ref{thm:main}.]
We proceed by induction on $\dim(X)$; the case when $\dim(X)=0$ is obvious since then $X(K)\cap\Gamma$ is a finite set and so, each of the groupless $F$-sets from our intersection are singletons (corresponding to $r=0$ in equation~\eqref{eq:F-orbits}).

Clearly, it suffices to assume $X$ is irreducible. Also, we may assume $X(K)\cap\Gamma$ is Zariski dense in $X$ since otherwise we could replace $X$ by the Zariski closure of $X(K)\cap\Gamma$ and be done by the inductive hypothesis.

We let $U:={\rm Stab}_G(X)$ be the stabilizer of $X$ in $G$. We have two possibilities depending on whether $U$ is finite, or not.

{\bf Case 1.} $\dim(U)>0$.

In this case, we let $\pi_0:G\lra G/U$ be the natural group homomorphism; in particular, $G_0:=G/U$ is a semiabelian variety defined over a finite field since $U$ is defined over a finite extension of $\Fq$. We let $\Gamma_0:=\pi_0(\Gamma)$ and also, let $X_0:=\pi_0(X)$. 

Since $\dim(U)>0$, then $\dim(X_0)<\dim(X)$ and so, by the inductive hypothesis, we have that $X_0(K)\cap \Gamma_0$ is a union of finitely many groupless $F$-sets $B_i$ in $\Gamma_0$ along with finitely many generalized $F$-sets $C_i$ in $\Gamma_0$. We have that
\begin{equation}
\label{eq:in}
X(K)\cap \Gamma=\pi_0^{-1}\left(X_0(K)\cap\Gamma_0\right)\cap \Gamma = \left(\pi_0|_{\Gamma}\right)^{-1}\left(X_0(K)\cap\Gamma_0\right).
\end{equation}   
Clearly, each $(\pi_0|_{\Gamma})^{-1}(B_i)$ is a generalized $F$-set in $\Gamma$ as in Definition~\ref{def:F1}. Now, each $C_i$ is a set of the form 
$$\left(f|_{\Gamma_0}\right)^{-1}(S_0)=f^{-1}(S_0)\cap\Gamma_0,$$
where $f:G_0\lra H$ is a surjective group homomorphism of semiabelian varieties over $K$ in which $\dim(\ker(f))>0$ and $H$ is defined over a finite extension of $\Fq$, and $S_0$ is a groupless $F$-set in $f(\Gamma_0)\subset H(K)$ as in Definition~\ref{def:F0}. 
So, using \eqref{eq:in}, along with the fact that  
\begin{equation}
\label{eq:inclusion}
\pi_0^{-1}\left(f^{-1}(S_0)\cap \Gamma_0\right)\cap \Gamma=\left(f\circ \pi_0\right)^{-1}(S_0)\cap \Gamma,
\end{equation}
then we obtain that $X(K)\cap \Gamma$ has the desired form as in the conclusion of Theorem~\ref{thm:main}. 

{\bf Case 2.} $U$ is finite.

In this case, we let $\tilde{\Gamma}$ be the $\mathbb{Z}[F]$-submodule spanned by $\Gamma$ inside $G(K)$; since $F$ is integral over $\mathbb{Z}$ (inside ${\rm End}(G)$), then  $\tilde{\Gamma}$ is also a finitely generated subgroup of $G(K)$. According to \cite{F-sets} (see Theorem~\ref{thm:R-T}), we have that 
\begin{equation}
\label{eq:in-3}
X(K)\cap\tilde{\Gamma}=\bigcup_{i=1}^\ell \left(S_i+\Gamma_i\right), 
\end{equation}
where each $S_i\subset \tilde{\Gamma}$ is a groupless $F$-set as in Definition~\ref{def:F0}, while each $\Gamma_i$ is a subgroup of $\tilde{\Gamma}$. Now, since 
\begin{equation}
\label{eq:in-2}
X(K)\cap\Gamma=\left(X(K)\cap\tilde{\Gamma}\right)\cap \Gamma,
\end{equation}
it suffices to prove that for each $i=1,\dots, \ell$, there exists a subset $A_i\subseteq X(K)\cap\Gamma$ which is a union of finitely many groupless $F$-sets in $\Gamma$ along with finitely many generalized $F$-sets in $\Gamma$ such that 
\begin{equation}
\label{eq:inside}
(S_i+\Gamma_i)\cap\Gamma\subseteq A_i;
\end{equation}
then combining equations~\eqref{eq:in-3},~\eqref{eq:in-2}~and~\eqref{eq:inside},  we would get that 
$$X(K)\cap\Gamma = \bigcup_{i=1}^\ell \left(S_i+\Gamma_i\right)\cap\Gamma = \bigcup_{i=1}^\ell A_i$$
is indeed a finite union of groupless $F$-sets in $\Gamma$ along with finitely many generalized $F$-sets in $\Gamma$, as claimed in the conclusion of Theorem~\ref{thm:main}. 

In order to prove the existence of a set $A_i$ (for each $i=1,\dots, \ell$) as in equation~\eqref{eq:inside}, we deal with two additional cases.

{\bf Case 2a.} $\Gamma_i$ is an infinite subgroup.

In this case, we let $X_i$ be the Zariski closure of $S_i+\Gamma_i$; clearly, $X_i\subseteq X$. We claim that $X_i$ is a proper subvariety of $X$. Indeed, by construction, $\Gamma_i\subseteq {\rm Stab}_G(X_i)$ and since $\Gamma_i$ is infinite, then we cannot have that $X_i=X$ because ${\rm Stab}_G(X)$ is finite. So, $\dim(X_i)<\dim(X)$ and by our inductive hypothesis, we have that that $A_i:=X_i(K)\cap \Gamma$ satisfies the conclusion from Theorem~\ref{thm:main}; therefore
$$\left(S_i+\Gamma_i\right)\cap\Gamma\subseteq A_i,$$
where $A_i$ is a union of finitely many groupless $F$-sets along with finitely many generalized $F$-sets, as desired for \eqref{eq:inside}.

{\bf Case 2b.} $\Gamma_i$ is finite.

In this case, letting $s:=\#\Gamma_i$, we have that $S_i+\Gamma_i$ is a union of $s$ groupless $F$-sets as in Definition~\ref{def:F0}. Now, \cite[Theorem~3.1]{TAMS} shows that the intersection of a groupless $F$-set with a finitely generated group is a itself a finite union of groupless $F$-sets; so, 
$$A_i:=(S_i+\Gamma_i)\cap\Gamma$$
is a finite union of groupless $F$-sets in $\Gamma$ as desired for \eqref{eq:inside}.

This concludes our proof of Theorem~\ref{thm:main}.  
\end{proof}

Theorem~\ref{thm:3} is an immediate corollary of our next two results which provide a more precise form of the intersection between a subvariety $X$ of $G$ with a finitely generated subgroup of $G(K)$ when $X$ is a curve, respectively when $G$ is a simple semiabelian variety.

\begin{proposition}
\label{prop:curve}
Let $G$ be a semiabelian variety defined over a finite subfield of an algebraically closed field $K$, let $\Gamma\subset G(K)$ be a finitely generated subgroup, and let $X\subseteq G$ be an irreducible curve. 
\begin{itemize}
\item[(i)] If $\dim\left({\rm Stab}_G(X)\right)>0$, then $X(K)\cap\Gamma$ is a coset of a subgroup of $\Gamma$. 
\item[(ii)] If ${\rm Stab}_G(X)$ is finite, then $X(K)\cap\Gamma$ is a finite union of groupless $F$-sets.
\end{itemize}
\end{proposition}

\begin{proof}
The proof of part~(i) is immediate since then $X=\gamma+G_1$, for some point $\gamma\in G(K)$ and some $1$-dimensional connected algebraic subgroup $G_1\subseteq G$. So, then the intersection $X(K)\cap\Gamma$ is simply a coset of the subgroup $G_1(K)\cap\Gamma$ of $\Gamma$.

Now, we assume ${\rm Stab}_G(X)$ is finite. Then we let $\tilde{\Gamma}$ be the $\mathbb{Z}[F]$-submodule of $G(K)$ spanned by $\Gamma$; by Theorem~\ref{thm:R-T}, we have that $\tilde{\Gamma}$ intersects $X(K)$ in a finite union of $F$-sets $S_i$ in $\tilde{\Gamma}$. But then at the expense of replacing each $S_i$ with finitely many other $F$-sets, we may assume that each such $F$-set is groupless  (see also the proof of {\bf Case 2b} in  Theorem~\ref{thm:main}). Finally, another application of \cite[Theorem~3.1]{TAMS} yields that each $S_i\cap\Gamma$ is a finite union of groupless $F$-sets in $\Gamma$, as desired. 
\end{proof}

\begin{proposition}
\label{prop:simple}
Let $G$ be a simple semiabelian variety defined over a finite subfield of an algebraically closed field $K$, let $\Gamma\subset G(K)$ be a finitely generated group, and let $X\subset G$ be a proper subvariety. Then $X(K)\cap\Gamma$ is a finite union of groupless $F$-sets in $\Gamma$. 
\end{proposition}

\begin{proof}
First of all, we note that if $\Gamma$ is a finite group, then clearly $X(K)\cap\Gamma$ is a finite set and thus a finite union of groupless $F$-sets, as desired.

So, from now on, we assume that $\Gamma$ is infinite. According to our Theorem~\ref{thm:main}, we know that $X(K)\cap\Gamma$ is a finite union of groupless $F$-sets in $\Gamma$ along with (possibly) finitely many generalized $F$-sets in $\Gamma$. Now, for any such generalized $F$-set in $\Gamma$ (call it $S$), we have that 
$$S=\left(\pi|_\Gamma\right)^{-1}(S_0),$$
where $\pi:G\lra H$ is a surjective group homomorphism of semiabelian varieties defined over a finite subfield of $K$, $S_0$ is a groupless $F$-set in $\pi(\Gamma)\subset H(K)$ and moreover, $\dim(\ker(\pi))>0$. But since $G$ is a simple semiabelian variety, this means that $\ker(\pi)=G$, i.e., $H$ is the trivial group variety and so, $S$ would have to be the entire subgroup $\Gamma$. But then its Zariski closure in $G$ is an infinite algebraic subgroup of $G$ (note that $\Gamma$ is assumed now to be infinite) and so, once again because $G$ is simple, we would conclude that $\Gamma$ is Zariski dense in $G$. But then because $S=\Gamma$ is contained in $X$, we would have that $X=G$, contradicting the fact that $X$ is a proper subvariety of $G$. Therefore, we have no generalized $F$-sets in $\Gamma$ contained in the intersection $X(K)\cap\Gamma$.

This concludes our proof for Proposition~\ref{prop:simple}.
\end{proof}

Theorem~\ref{thm:main_2} follows easily from our Theorem~\ref{thm:main} combined with Theorem~\ref{thm:H}.

\begin{proof}[Proof of Theorem~\ref{thm:main_2}.]
Clearly, as argued in the proof of Theorem~\ref{thm:main}, it suffices to prove Theorem~\ref{thm:main_2} assuming that $X$ is an irreducible subvariety of $G$ and $X(K)\cap\Gamma$ is Zariski dense in $X$. Then Theorem~\ref{thm:H} yields that
$$X=\gamma+\pi^{-1}(X_0),$$
where $\pi:G_0\lra H$ is a surjective group homomorphism of semiabelian varieties, while $G_0$ is a semiabelian subvariety of $G$ and $\gamma\in G(K)$; moreover, $H$ and the subvariety $X_0\subseteq H$ are defined over over a finite subfield $\F_q\subset K$. Then for $x\in G(K)$, we have $x\in X(K)$ if and only if ``$x-\gamma\in G_0(K)$ and $\pi(x-\gamma)\in X_0(K)$". We denote $\Gamma_0=G_0(K)\cap\Gamma$.

Pick $x_0\in X(K)\cap\Gamma$. Let $g_0=x_0-\gamma\in G_0(K)$. We have $x_0+\Gamma_0=(\gamma+G_0(K))\cap\Gamma$. As a result, for any $x\in\Gamma$, we have $x-\gamma\in G_0(K)$ if and only if there exists $\gamma_0\in\Gamma_0$ such that $x=x_0+\gamma_0$. Thus $x-\gamma=g_0+\gamma_0$ and so, $\pi(x-\gamma)\in X_0(K)$ yields $\pi(\gamma_0)\in-\pi(g_0)+X_0(K)$.

Let $X_0'=-\pi(g_0)+X_0$ which is a subvariety of $H$. The discussion above implies that $X(K)\cap\Gamma=x_0+(\pi|_{\Gamma_0})^{-1}(X_0'(K)\cap\pi(\Gamma_0))$. So, considering the subvariety $X_0'\subseteq H$, along with the finitely generated subgroup $\pi(\Gamma_0)$ of $H(K)$, then we apply Theorem~\ref{thm:main} to conclude that the intersection $X_0'(K)\cap \pi(\Gamma_0)$ is a finite union of generalized $F$-sets in $\pi(\Gamma_0)$ along with finitely many groupless $F$-sets in $\pi(\Gamma_0)$. But whether $S$ is a generalized $F$-set in $\pi(\Gamma_0)$ or a groupless $F$-set in $\pi(\Gamma_0)$, $x_0+(\pi|_{\Gamma_0})^{-1}(S)$ will always be a pseudo-generalized $F$-set in $\Gamma$ (see also Remark~\ref{rem:less}). This shows that $X(K)\cap\Gamma$ is a finite union of pseudo-generalized $F$-sets in $\Gamma$, as desired.
\end{proof}


\textbf{Acknowledgement.} The second author is very grateful to his advisor Junyi Xie who introduced this topic to him.

The first author is supported by a Discovery NSERC grant, while the second author is supported by an NSFC Grant (No. 12271007).


\end{document}